\documentclass[a4paper]{amsart}
\usepackage{amsmath,amssymb,amsthm,enumerate}
\title[The Bonnet theorem for statistical manifolds]
{The Bonnet theorem for statistical manifolds}
      
\author{TAIJI MARUGAME}
\date{}

\newcommand\R{\mathbb{R}}

\renewcommand\a{\alpha}

\renewcommand\d{\delta}

\newcommand\pa{\partial}

\newtheorem{lem}{Lemma}[section]

\newtheorem{theorem}[lem]{Theorem}
\newtheorem{prop}[lem]{Proposition}

\theoremstyle{definition}
\newtheorem{dfn}[lem]{Definition}

\numberwithin{equation}{section}

\makeatletter
\@addtoreset{equation}{section}

\makeatother
\address{Mathematical Analysis Team, RIKEN Center for Advanced Intelligence Project (AIP), 1-4-1 Nihonbashi, Chuo-ku, Tokyo 103-0027, Japan}
\address{Department of Mathematics, Graduate School of Science, Osaka University, 1-1 Machikaneyama-cho Toyonaka Osaka 560-0043, Japan}
\email{taiji.marugame@riken.jp}


\begin{document}

\begin{abstract} 
We prove the Bonnet theorem for statistical manifolds, which states that if a statistical manifold admits tensors satisfying the Gauss--Codazzi--Ricci equations, then it is locally embeddable to a flat statistical manifold (or a Hessian manifold). The proof is based on the notion of statistical embedding to the product of a vector space and its dual space introduced by Lauritzen. As another application of Lauritzen's embedding, we show that a statistical manifold admitting an affine embedding of codimension 1 or 2 is locally embeddable to a flat statistical manifold of the same codimension.
\end{abstract}
\maketitle

%%%%%%%%%%%%%%%%%%%%%%%%%%%%%%%%%%%%%%%%%%%%%%%%%%%%%%%%%%%%%%%
\section{Introduction}
A statistical manifold $(M, g, \nabla)$ is a $C^\infty$-manifold endowed with a Riemannian metric $g$ and a torsion-free affine connection $\nabla$ such that $\nabla g$ is totally symmetric. In information geometry, statistical structures naturally appear as the 
geometric structure on the space of probability distributions and play an important role 
in the geometric treatment of statistical problems such as statistical inference (see \cite{Amari, AN, AJLS}). 

From the viewpoint of differential geometry, statistical manifolds can be considered as generalization of Riemannian manifolds since a Riemannian manifold with the Levi-Civita connection $(M, g, \nabla^g)$ is an example of a statistical manifold. However, there arises a crucial difference when we consider the notion of ``flatness''. A statistical manifold $(M, g, \nabla)$ is called flat when $\nabla$ is a flat connection; we do not assume that $g$ is a flat metric. Then, the symmetry property of $\nabla g$ imposes that $g$ is locally the Hessian of a function: $g=\nabla^2 \psi$. Thus, a flat statistical manifold is equivalent to a so-called Hessian manifold, and its local structure is not unique unlike the flat models in Riemannian or other Cartan geometries.

In Riemannian geometry, a fundamental problem is the isometric embedding problem of  Riemannian manifolds into the Euclidean spaces, and one can pose the corresponding problem for statistical manifolds. H. V. L\^e \cite{Le} solved this problem and proved that any statistical manifold can be locally embedded to a flat statistical manifold of sufficiently large dimension; see also \cite{AJLS}. This is a statistical analogue of the Nash embedding theorem in Riemannian geometry. 

In this paper, we consider a statistical analogue of the Bonnet theorem, which is a fundamental theorem in the submanifold theory in Riemannian geometry. A submanifold in the Euclidean space inherits the induced metric (the first fundamental form) and an extrinsic curvature called the second fundamental form. These quantities satisfy algebraic and differential equations called the Gauss--Codazzi equations. The classical Bonnet theorem asserts that conversely if a Riemannian manifold admits tensors with the Gauss--Codazzi equations, then it is locally embedded isometrically to the Euclidean space. For a submanifold of a flat statistical manifold, we can also define the induced statistical structure $(g, \nabla)$ and extrinsic curvature quantities $h^a_{ij}, h^*_{aij}, \tau^a_{bi}$ with the Gauss--Codazzi--Ricci equations; see \eqref{GCR}. The goal of this paper is to prove the following Bonnet theorem for statistical manifolds:

\begin{theorem}\label{Bonnet}
Let $(M, g, \nabla)$ be an $n$-dimensional statistical manifold and $r$ a positive integer. If there exist tensors 
\[
(h^a_{ij})\in\Gamma(S^2T^*M\otimes\underline\R^r), \quad 
(h^*_{aij})\in\Gamma(S^2T^*M\otimes(\underline\R^r)^*), \quad
(\tau^a_{bi})\in\Gamma(T^*M\otimes\underline\R^r\otimes(\underline\R^r)^*)
\]
which satisfy the Gauss--Codazzi--Ricci equations \eqref{GCR}, then $M$ is locally embeddable to a flat statistical manifold of dimension $n+r$.
\end{theorem}
As noted above, the local structure of flat statistical manifolds is not unique and this prevents us from applying straightforwardly the standard technique of proving Bonnet-type theorems. We overcome this difficulty by invoking the notion of statistical embedding 
to the product of mutually dual vector space $V\times V^*$, introduced by Lauritzen \cite{Lau}. A key step in the proof of Theorem \ref{Bonnet} is to establish that the embeddability to $V\times V^*$ in the sense of Lauritzen is in fact locally equivalent to the embeddability to a flat statistical manifold (Theorem \ref{equivalence}).
This equivalence has an immediate application to the embeddability of statistical structures $(g, \nabla^{(\a)})\ (\a\in\R)$; see Proposition \ref{alpha}.

We also have an application of Theorem \ref{equivalence} to affine embeddings of statistical manifolds. There have been studies of statistical manifolds from the viewpoint of affine differential geometry (\cite{NS}). There is a notion of embedding of statistical manifolds to affine spaces of codimension $1, 2$ and the affine embeddability is closely related to the generalized conformal structures on statistical manifolds (\cite{DNV, Kurose, Matsuzoe}). However, in the affine embedding, the metric $g$ is defined by the second fundamental form instead of the induced metric, so its relation to the usual statistical embedding is not clear at once. In this paper, we prove 
that the affine embeddability implies the embeddablity to a flat statistical manifold:

\begin{theorem}\label{affine}
Let $r=1, 2$. If an $n$-dimensional statistical manifold $(M, g, \nabla)$ admits a local affine embedding to $\R^{n+r}$, then it is locally embeddable to a flat statistical manifold of dimension $n+r$.
\end{theorem}
The proof is given by the application of Theorem \ref{equivalence} to the pair of the affine embedding and its dual (conormal) embedding. 
\bigskip

This paper is organized as follows: In \S \ref{statistical}, we review some notions on statistical manifolds and prove the equivalence (Theorem \ref{equivalence}) of the usual embeddability to a flat statistical manifold and that of Lauritzen. Then, we apply this fact to affine embeddings and prove Theorem \ref{affine}. In \S\ref{Bonnet-theorem}, we define extrinsic curvature quantities on submanifolds in a flat statistical manifold and derive the Gauss--Codazzi--Ricci equations. Finally, we prove the Bonnet theorem (Theorem \ref{Bonnet}) based on Theorem \ref{equivalence}.

{\it Notations}: We adopt Einstein's summation convention. The indices $i, j, k, l$ are used for tensors on the (sub)manifold $M$ and run from $1$ to $n$. The indices $A, B, K, L$ are used for tensors on the ambient space $\mathcal{S}$ or $V$ and run from $1$ to $N(=n+r)$. The indices $a, b, c$ are used for complementary directions to $M$ and run from $1$ to $r (=N-n)$. The skew symmetrization of indices are denoted by $[\ ,\ ]$ and performed only over indices such as $i, j, k$. For example,  
\[
h^a_{l[i}h^*_{bj]}{}^l\quad  {\rm means}\quad \frac{1}{2}(h^a_{li}h^*_{bj}{}^l-h^a_{lj}h^*_{bi}{}^l).
\]

\bigskip\noindent {\bf Acknowledgment} The author is grateful to Professor Hiroshi Matsuzoe for comments and information on related literatures.
This work was partially supported by JSPS KAKENHI Grant Number JP20K22318.

%%%%%%%%%%%%%%%%%%%%%%%%%%%%%%%%%%%%%%%%%%%%%%%%%%%%%%%%%%%%%%%%%%%%%%%

\section{Statistical manifolds}\label{statistical}
\subsection{Statistical structures}
We will briefly review some basic notions on statistical manifolds; we refer the reader to the books \cite{AN, AJLS} for details. 

Let $M$ be an $n$-dimensional $C^\infty$-manifold. We call $(g, \nabla)$ a {\it statistical structure} if $g$ is a Riemannian metric and $\nabla$ is a torsion-free affine connection on $M$ such that $\nabla g$ is a totally symmetric $3$-tensor: $\nabla_{[i}g_{j]k}=0$. We lower and raise the indices by the metric $g_{ij}$ and its inverse $g^{ij}$. The {\it dual connection} $\nabla^*$ of $\nabla$ is the affine connection characterized by the equation
\[
X\cdot g(Y, Z)=g(\nabla_X Y, Z)+g(Y, \nabla^*_X Z), \quad X, Y, Z\in\Gamma(TM).
\]
The condition that $\nabla g$ is totally symmetric is equivalent to saying that $\nabla^*$ is also torsion-free. Thus, $(g, \nabla^*)$ also becomes a statistical structure, which is called the {\it dual statistical structure}. 
Since $\nabla$ and $\nabla^*$ are torsion-free, their curvature tensors are given by the Ricci identities
\[
(\nabla_i\nabla_j-\nabla_j\nabla_i)X^k=R_{ij}{}^k{}_l X^l, \quad 
(\nabla^*_i\nabla^*_j-\nabla^*_j\nabla^*_i)X^k=R^*_{ij}{}^k{}_l X^l.
\]
It is easy to show that $R$ and $R^*$ are related by
\begin{equation}\label{R-relation}
R_{ijkl}=-R^*_{ijlk}.
\end{equation}

We also have a 1-parameter family of statistical structure $(g, \nabla^{(\a)}), \a\in\R$, where the $\a$-{\it connection} $\nabla^{(\a)}$ is defined by
\[
\nabla^{(\a)}:=\frac{1+\a}{2}\nabla+\frac{1-\a}{2}\nabla^*.
\]
The dual connection of $\nabla^{(\a)}$ is equal to $\nabla^{(-\a)}$ and since both connections are torsion-free, $(g, \nabla^{(\a)})$ is indeed a statistical structure.
\subsection{Flat statistical structures}
A statistical structure $(g, \nabla)$ is said to be {\it flat} when $\nabla$ is a flat connection: $R_{ij}{}^k{}_l=0$. In this case, the dual connection is also flat by the equation \eqref{R-relation}, so we also call $(g, \nabla, \nabla^*)$ a {\it dually flat structure}. Although the metric $g$ may not be flat, the condition $\nabla_{[i}g_{j]k}=0$ and the Poincar\'e lemma imply that locally one can write as
\[
g_{ij}=\nabla_i\nabla_j \psi
\]
with some function $\psi$. Conversely, if $\nabla$ is flat and $g$ is locally written as the Hessian of a function, then $(g, \nabla)$ becomes a flat statistical structure. Thus, the flat statistical structure is equivalent to the {\it Hessian structure}.

Let $(\xi^1, \dots, \xi^n)$ be local affine coordinates for $\nabla$, and $\psi$ a Hesse potential of $g$:
\[
\nabla\frac{\pa}{\pa \xi^i}=0, \quad g_{ij}=\frac{\pa^2\psi}{\pa\xi^i\pa \xi^j}. 
\]
The {\it Legendre transform} of $(\xi^i, \psi)$ is defined by
\[
\eta_i:=\frac{\pa\psi}{\pa\xi^i}, \quad \psi^*:=\xi^i\eta_i-\psi.
\]
Then, $(\eta_1, \dots, \eta_n)$ become affine coordinates for $\nabla^*$, called the {\it dual coordinates}, and $\psi$ gives a Hesse potential of $g$ with respect to $\nabla^*$. Moreover, we have
\[
\frac{\pa \psi^*}{\pa\eta_i}=\xi^i, \quad \Bigl(\frac{\pa^2 \psi^*}{\pa\eta_i\pa\eta_j}\Bigr)=\Bigl(\frac{\pa \xi^i}{\pa \eta_j}\Bigr)=\Bigl(\frac{\pa^2\psi}{\pa\xi^i\pa \xi^j}\Bigr)^{-1}.
\]
In terms of $\xi^i$ and $\eta_i$, the metric $g$ has the expression
\begin{equation*}
g=\frac{\pa^2\psi}{\pa\xi^i\pa \xi^j}d\xi^i\otimes d\xi^j=\frac{\pa^2 \psi^*}{\pa\eta_i\pa\eta_j}d\eta_i\otimes d\eta_j=d\xi^i\otimes d\eta_i.
\end{equation*}
\subsection{Statistical submanifolds and statistical embeddings}
Let $(\mathcal{S}, G, D)$ be a statistical manifold and $M\subset \mathcal{S}$ a submanifold.
Then, we can define a statistical structure $(g, \nabla)$ on $M$ by setting
\[
g(X, Y)=G(X, Y), \quad g(\nabla_X Y, Z)=G(D_X Y, Z)
\]
for $X, Y, Z\in \Gamma(TM)$. In other words, the metric $g$ is defined by the restriction of $G$ and the connection $\nabla$ is the ``orthogonal projection'' of $D$ with respect to $G$. One can see that $\nabla$ is torsion-free and the dual connection $\nabla^*$ is given by $g(\nabla^*_X Y, Z)=G(D^*_X Y, Z)$. Thus, $\nabla^*$ is also torsion-free, so $(g, \nabla)$ is a statistical structure.

Let $(M, g, \nabla)$, $(\mathcal{S}, G, D)$ be two statistical manifolds, and $f\colon M\longrightarrow \mathcal{S}$ an embedding. We say $f$ is a {\it statistical embedding} if $(g, \nabla)$ agrees with the statistical structure induced on $f(M)\subset\mathcal{S}$ when we identify $M$ with its image $f(M)$. In this paper, we only consider local embeddings to flat statistical manifolds. When we say $(M, g, \nabla)$ is {\it locally embeddable to a flat statistical manifold}, we mean that for any point $x\in M$, there exist an open neighborhood $U$ of $x$, a flat statistical manifold $(\mathcal{S}, G, D)$, and a statistical embedding $f\colon U\longrightarrow \mathcal{S}$. 
\subsection{Lauritzen's statistical embedding}
Lauritzen \cite{Lau} introduced another notion of statistical embedding whose target space is the product of two mutually dual vector spaces instead of a flat statistical manifold (see also \cite{AN}):
\begin{dfn}Let $(M, g, \nabla)$ be a statistical manifold, and $V$ a real vector space. 
A pair of embeddings
\[
(f, \varphi)\colon M\longrightarrow V\times V^*
\]
is called a {\it statistical embedding (in the sense of Lauritzen)} if 
\begin{equation}\label{Lauritzen}
g(X, Y)=\langle f_*X, \varphi_* Y\rangle, \quad g(\nabla_X Y, Z)=\langle X\cdot f_*Y, \varphi_* Z\rangle
\end{equation}
hold for any $X, Y, Z\in\Gamma(TM)$, where $\langle\ ,\ \rangle$ denotes the natural paring of $V$ and $V^*$.
\end{dfn}
In the equation \eqref{Lauritzen}, we are identifying each tangent space $T_\xi V$ or $T_\eta V^*$ with 
$V$ or $V^*$, so $f_* X$ and $\varphi_* Y$ are regarded as $V$-valued or $V^*$-valued functions on $M$. In the index notation, we can express \eqref{Lauritzen} as
\[
g_{ij}=(\pa_i f^A)(\pa_j \varphi_A), \quad \Gamma_{ijk}\,(=g_{kl}\Gamma_{ij}{}^l)=(\pa_i\pa_j f^A)(\pa_k\varphi_A).
\]
We note that the mapping
\[
(\varphi, f)\colon M\longrightarrow V^*\times V
\]
gives a statistical embedding of the dual statistical manifold $(M, g, \nabla^*)$.

The following theorem asserts that the two notions of statistical embedding are locally equivalent: 

\begin{theorem}\label{equivalence}
Let $(M, g, \nabla)$ be an $n$-dimensional statistical manifold. Then the following are equivalent: 
\begin{itemize}
\item[(i)] $M$ is locally embeddable to an $N$-dimensional flat statistical manifold. \\
\item[(ii)] $M$ is locally embeddable to $V\times V^*$ in the sense of Lauritzen for an $N$-dimensional vector space $V$. 
\end{itemize}
\end{theorem}
\begin{proof}
First we prove (i)$\Rightarrow$(ii). Let $(\mathcal{S}, G, D)$ be an $N$-dimensional flat statistical manifold, and suppose $F\colon M\longrightarrow \mathcal{S}$ is a statistical embedding defined on a neighborhood of a point.  (For simplicity, we also denote the neighborhood by $M$.) We take local affine coordinates $(\xi^A)$ for $D$ and a Hesse potential $\psi$ for $G$, and denote the dual coordinates by $(\eta_A)$:  
\[
G=\frac{\partial^2 \psi}{\partial \xi^A\partial \xi^B}d\xi^A \otimes d\xi^B, \quad 
\eta_A=\frac{\partial \psi}{\partial \xi^A}.
\]
Let $V:=\mathbb{R}^N$ and define $(f, \varphi)\colon M\longrightarrow V\times V^*$ by $f^A:=\xi^A\circ F$ and $\varphi_A:=\eta_A\circ F$. Then, since $G=d\xi^A\otimes d\eta_A$ we have $g=F^* G=df^A\otimes d\varphi_A$, i.e., $g_{ij}=(\partial_i f^A)(\partial_j\varphi_A)$. Moreover, we have 
\begin{align*}
g(\nabla_{\pa_i}\pa_j, \pa_k)&=G(D_{F_*\pa_i}(F_*\pa_j), F_*\pa_k) \\
&=G\Bigl((\pa_i\pa_j f^A)\frac{\pa}{\pa \xi^A}, (\pa_k\varphi_B)\frac{\pa}{\pa \eta_B}\Bigr) \\
&=(\pa_i\pa_j f^A)(\pa_k\varphi_A).
\end{align*}
Thus, $(f, \varphi)$ is a statistical embedding in the sense of Lauritzen.

Next we prove (ii)$\Rightarrow$(i). Suppose that $(f, \varphi)\colon M\longrightarrow V\times V^*$ is a local statistical embedding. We will construct a local flat statistical structure $(G, D)$ on $V$ so that $f$ becomes a statistical embedding. Let $(f^A, \varphi_A)$ be the components with respect to linear coordinates $(\xi^A)$ of $V$ and its dual coordinates. We consider the 1-form $\omega:=\langle df, \varphi\rangle=(\pa_i f^A)\varphi_A dx^i$ on $M$. Since $\pa_i f^A\pa_j\varphi_A=g_{ij}$, we have 
\begin{align*}
d\omega=\bigl((\pa_j\pa_i f^A)\varphi_A+\pa_i f^A\pa_j \varphi_A\bigr)dx^j\wedge dx^i =0.
\end{align*}
Thus, there (locally) exists a function $\psi_0\in C^\infty(M)$ such that $d\psi_0=\omega$. If we regard $M$ as a submanifold of $V$ and consider $\varphi_Ad\xi^A$ as a section of $T^*V|_M$, it holds that $(\varphi_A d\xi^A)|_{TM}=d\psi_0$. Hence we can construct a local function $\psi\in C^\infty(V)$ which satisfies
\begin{equation}\label{psi-cond}
\psi|_M=\psi_0, \quad \frac{\pa \psi}{\pa \xi^A}\Big|_{M}=\varphi_A
\end{equation}
by suitably choosing the 1-jet of $\psi$ along $M$ in the transversal directions. Moreover, we can take such a $\psi$ so that the Hessian $(\pa^2\psi/\pa\xi^A\pa\xi^B)$ is (locally) positive definite as follows: Let $\rho_1, \dots, \rho_r\in C^\infty(V)\ (r=N-n)$ be local defining functions of $M$, i.e., locally $M=\{\rho_1=\cdots=\rho_r=0\}$, $d\rho_1\wedge\cdots \wedge d\rho_r\neq0$. We modify $\psi$ as 
\[
\psi'=\psi+C(\rho_1^2+\cdots+\rho_r^2)
\]
with a constant $C>0$. Then, $\psi'$ still satisfies \eqref{psi-cond}, and by differentiating the equation
\[
(\pa_j f^A)\varphi_A=\pa_j\psi_0=\frac{\pa \psi'}{\pa \xi^A}\pa_j f^A
\]
on $M$, we have 
\[
(\pa_i\pa_j f^A)\varphi_A+\pa_j f^A\pa_i\varphi_A=\frac{\pa^2\psi'}{\pa\xi^A\pa\xi^B}\pa_i f^A\pa_j f^B+\frac{\pa \psi'}{\pa \xi^A}\pa_i\pa_j f^A.
\]
Since $\pa_i f^A\pa_j\varphi_A=g_{ij}$ and $(\pa \psi'/\pa\xi^A)|_M=\varphi_A$, we obtain
\begin{equation}\label{Hesse-psi}
\frac{\pa^2\psi'}{\pa\xi^A\pa\xi^B}\pa_i f^A\pa_j f^B=g_{ij}.
\end{equation}
This implies that the Hessian of $\psi'$ is positive definite on $TM$, so 
we can make it positive definite on a neighborhood of $M$ in $V$ by taking a sufficiently large $C$.

Now we define a local flat statistical structure on $V$ by the metric 
\[
G:=\frac{\pa^2\psi'}{\pa\xi^A\pa\xi^B}d\xi^A\otimes d\xi^B
\]
and the flat connection $D$ associated with the coordinates $(\xi^A)$. The equation \eqref{Hesse-psi} implies $f^*G=g$. Moreover, we have
\begin{align*}
G(D_{f_*\pa_i}(f_*\pa_j), f_*\pa_k) 
&=\Bigl[d\xi^A\otimes d\Bigl(\frac{\pa\psi'}{\pa \xi^A}\Bigr)\Bigr]
\Bigl(\pa_i\pa_j f^A \frac{\pa}{\pa \xi^A}, f_*\pa_k\Bigr) \\
&=(\pa_i\pa_j f^A)(\pa_k\varphi_A) \\
&=g(\nabla_{\pa_i}\pa_j, \pa_k).
\end{align*}
Thus, $f\colon M\longrightarrow V$ is a statistical embedding.
\end{proof}
As an application of Theorem \ref{equivalence}, we will prove the following
\begin{prop}\label{alpha}
If a statistical manifold $(M, g, \nabla)$ is locally embeddable to an $N$-dimensional flat statistical manifold, then $(M, g, \nabla^{(\a)})$, $\a\in\R$, are locally embeddable to a flat statistical manifold of dimension $2N$. 
\end{prop}
\begin{proof}
By Theorem \ref{equivalence}, we have a local statistical embedding 
$(f, \varphi)\colon (M, g, \nabla)\longrightarrow V\times V^*$ for an $N$-dimensional vector space $V$. For each $\a\in\R$, we define a mapping 
\[
(F, \Phi)\colon M\longrightarrow (V\oplus V^*)\times (V^*\oplus V)
\]
by
\[
F:=\Bigl(f, \frac{1-\a}{2}\varphi\Bigr), \quad \Phi:=\Bigl(\frac{1+\a}{2}\varphi, f\Bigr).
\]
Then, $F$ and $\Phi$ are local embeddings and we have
\begin{align*}
\langle F_*X, \Phi_* Y\rangle&=\frac{1+\a}{2}g(X, Y)+\frac{1-\a}{2}g(Y, X)=g(X, Y), \\
\langle X\cdot F_* Y, \Phi_* Z\rangle&=\frac{1+\a}{2}g(\nabla_X Y, Z)+\frac{1-\a}{2}g(\nabla^*_X Y, Z)=g(\nabla^{(\a)}_X Y, Z).
\end{align*}
Hence $(F, \Phi)$ is a statistical embedding, and by Theorem \ref{equivalence}, $(M, g, \nabla^{(\a)})$ is locally embeddable to a flat statistical manifold of dimension $\dim(V\oplus V^*)=2N$.
\end{proof}

\subsection{Affine embeddings of statistical manifolds: proof of Theorem \ref{affine}}
We can also use Theorem \ref{equivalence} to prove Theorem \ref{affine}, which gives a relation between statistical embeddings and affine embeddings. 

Let $M$ be an $n$-dimensional $C^\infty$-manifold. First we consider the affine embedding of codimension $1$. Let $f\colon M\longrightarrow \R^{n+1}$ be an embedding, and $\xi\in\Gamma(f^*T\R^{n+1})$ a transverse vector field. Then, according to the decomposition $f^*T\R^{n+1}=f_*TM\oplus\R\xi$, we can write as
\begin{align*}
D_X(f_*Y)&=f_*(\nabla_X Y)-g(X, Y)\xi, \\
D_X \xi&=f_*(S(X))+\tau(X)\xi
\end{align*}
for $X, Y\in\Gamma(TM)$, where $D$ denotes the canonical flat connection on $f^*T\R^{n+1}$. One can see that $\nabla$ defines a torsion-free affine connection on $M$, and $g$ gives a symmetric $2$-tensor, called the {\it affine second fundamental form}.
The endomorphism $S$ is called the {\it affine shape operator}. These quantities satisfy the Gauss--Codazzi--Ricci equations derived from the flatness of $D$, and it follows that $(g, \nabla)$ gives a statistical structure on $M$ if $g$ is positive definite and $\tau=0$ (equiaffine condition). In this case, we say the statistical manifold $(M, g, \nabla)$ admits an affine embedding to $\R^{n+1}$. Dillen--Nomizu--Vrankan \cite{DNV} proved the affine Bonnet theorem, which states that if a statistical manifold has tensors with the Gauss--Codazzi--Ricci equations, then it admits
an affine embedding to $\R^{n+1}$. Moreover, they showed that the embeddability is also characterized by the condition that the statistical structure is {\it $1$-conformally flat}, or equivalently the dual connection $\nabla^*$ is projectively flat; see also \cite{Kurose}.

When a statistical manifold $M$ admits an affine embedding $\{f, \xi\}$ as above, it also has the {\it dual embedding} or the {\it conormal mapping} $\varphi\colon M\longrightarrow (\R^{n+1})^*$ defined by
\[
\varphi(x)|_{f_*T_xM}=0, \quad \langle \xi_x, \varphi(x)\rangle=1.
\]
Since $\tau=0$, $\varphi$ satisfies
\begin{equation}\label{xi-phi}
\langle \xi, \varphi_* Z\rangle=Z\cdot\langle\xi, \varphi\rangle-\langle D_Z\xi, \varphi\rangle=0\quad  {\rm for}\ Z\in\Gamma(TM).
\end{equation}
Using this, we have
\begin{equation}\label{affine-Lauritzen}
\begin{aligned}
\langle f_*X, \varphi_*Y\rangle&=Y\cdot \langle f_*X, \varphi \rangle-\langle D_Y (f_*X), \varphi\rangle=g(X, Y), \\
\langle X\cdot f_*Y, \varphi_*Z\rangle&=\langle f_*(\nabla_X Y), \varphi_*Z\rangle
=g(\nabla_X Y, Z).
\end{aligned}
\end{equation}
Note that by the first equation, $\varphi_*$ is injective so $\varphi$ is a local embedding. These two equations imply that the pair $(f, \varphi)\colon M\longrightarrow \R^{n+1}\times (\R^{n+1})^*$ is a statistical embedding in the sense of Lauritzen. This proves Theorem \ref{affine} in the case of codimension $1$.  

Next we consider affine embeddings of codimension $2$. Let $f\colon M\longrightarrow \R^{n+2}$ be an embedding of an $n$-dimensional $C^\infty$-manifold. We assume that the position vector field $\eta=x^A(\pa/\pa x^A)\in \Gamma(f^*T\R^{n+2})$ is transverse to $f_*TM$. Supplying one more transverse vector field $\xi\in\Gamma(f^*T\R^{n+2})$, we consider the decomposition 
$f^*T\R^{n+2}=f_*TM\oplus\R\xi\oplus\R\eta$. For $X, Y\in\Gamma(TM)$, we write as 
\begin{align*}
D_X (f_*Y)&=f_*(\nabla_X Y)-g(X, Y)\xi-k(X, Y)\eta, \\
D_X \xi&=f_*(S(X))+\tau(X)\xi+\mu(X)\eta, \\
D_X\eta&=X.
\end{align*}
Again, $(g, \nabla)$ defines a statistical structure on $M$ if $g$ is positive definite and $\tau=0$, and we say that the statistical manifold $(M, g, \nabla)$ admits an affine embedding of codimension $2$. Matsuzoe \cite{Matsuzoe} proved that the affine embeddability of codimension $2$ is characterized by the condition that $(g, \nabla)$ is {\it conformally-projectively flat}, which is a generalization of the $1$-conformally flatness.

For an affine embedding $f$ of codimension $2$, we can also define the dual embedding $\varphi\colon M\longrightarrow (\R^{n+2})^*$ by
\[
\varphi(x)|_{f_*T_xM\oplus\R\eta_x}=0, \quad \langle \xi_x, \varphi(x)\rangle=1.
\]
Then, we again have the equation \eqref{xi-phi} and 
\[
\langle \eta, \varphi_* Z\rangle=Z\cdot\langle\eta, \varphi\rangle-\langle D_Z\eta, \varphi\rangle=0\quad  {\rm for}\ Z\in\Gamma(TM).
\]
Consequently, we obtain \eqref{affine-Lauritzen} as well, and $(f, \varphi)\colon M\longrightarrow \R^{n+2}\times (\R^{n+2})^*$ gives a statistical embedding in the sense of Lauritzen. Thus we have proved Theorem \ref{affine} in the case of codimension $2$.

%%%%%%%%%%%%%%%%%%%%%%%%%%%%%%%%%%%%%%%%%%%%%%%%%%%%%%%%%%%%%%%%%

\section{The Bonnet theorem}\label{Bonnet-theorem}
\subsection{The Gauss--Codazzi--Ricci equations for statistical embeddings}
Let $(M, g, \nabla)$ be an $n$-dimensional statistical manifold and $(\mathcal{S}, G, D)$ an $N$-dimensional flat statistical manifold. For a statistical embedding $f\colon M\longrightarrow \mathcal{S}$, we will define extrinsic curvature quantities and differential equations satisfied by them. 

We set $r:=N-n$ and let $(\nu_a)=(\nu_1, \dots, \nu_r)$ be a local orthonormal frame for $(f_*TM)^\perp\subset f^*T\mathcal{S}$ with respect to $G$. Then, locally we have the orthogonal decomposition
\[
f^*T\mathcal{S}=f_*TM\oplus \underline\R^r,
\]
where $\underline\R^r:=M\times \R^r$ is the trivial bundle defined by $(\nu_a)$. In calculations with the index notation, it is convenient to identify $\R^r$ with its dual $(\R^r)^*$ and regard $(\nu_a)$ as a frame of $(\underline\R^r)^*$ as well as $\underline\R^r$; thus we also write $\nu^a$ for $\nu_a$ under this identification.   

For vector fields $X, Y\in\Gamma(TM)$, we can write as
\begin{align*}
D_X Y&=f_*(\nabla_X Y)-h^a(X, Y)\nu_a, \\
D_X \nu_a&=f_*(S_a(X))+\tau^b_a(X)\nu_b
\end{align*}
and
\begin{align*}
D^*_X Y&=f_*(\nabla^*_X Y)-h^*_a(X, Y)\nu^a, \\
D^*_X \nu^a&=f_*(S^{*a}(X))+\tau^{*a}_b(X)\nu^b.
\end{align*}
The $\R^r$-valued or $(\R^r)^*$-valued symmetric $2$-tensors $h^a_{ij}, h^*_{aij}$ are called the {\it second fundamental forms} or the {\it embedding curvatures} in the literature \cite{Amari, Vos}. The {\it shape operator} $S_a$ satisfies
\[
g(S_a(X), Y)=G(D_X \nu_a, Y)=-G(\nu_a, D^*_X Y)=h^*_a(X, Y)
\]
and the similar equation holds for $S^{*a}$. Hence we have
\[
S_{ai}{}^j=h^*_{ai}{}^j, \quad S^{*a}{}_{i}{}^j=h^a{}_i{}^j.
\]
We also have
\[
\tau^{*a}_b(X)=G(D^*_X\nu^a, \nu_b)=-G(\nu^a, D_X\nu_b)=-\tau^a_b(X).
\]
Thus, the extrinsic tensors for the statistical embedding $f$ are given by
\[
h^a_{ij}, \quad h^*_{aij}, \quad \tau^a_{bi}.
\]

We will derive the Gauss--Codazzi--Ricci equations for these tensors from the flatness of $D$. Let $\boldsymbol{\nabla}:=f^*D$, $\boldsymbol{\nabla}^*:=f^*D^*$ be 
the flat connections on the vector bundle 
\[
E:=TM\oplus\underline\R^r\cong f^*T\mathcal{S}
\]
induced by $D, D^*$.
Then, they can be expressed as
\begin{equation}\label{conn-E}
\begin{aligned}
\boldsymbol{\nabla}_i
\begin{pmatrix}
X^j \\
\lambda^a
\end{pmatrix}
&=\begin{pmatrix}
\nabla_iX^j+h^*_{bi}{}^j\lambda^b \\
\nabla_i\lambda^a+\tau^a_{bi}\lambda^b-h^a_{ik}X^k
\end{pmatrix}, \\
\boldsymbol{\nabla}^*_i
\begin{pmatrix}
X^j \\
\lambda_a
\end{pmatrix}
&=\begin{pmatrix}
\nabla^*_iX^j+h^b{}_i{}^j\lambda_b \\
\nabla^*_i\lambda_a-\tau^b_{ai}\lambda_b-h^*_{aik}X^k
\end{pmatrix},
\end{aligned}
\end{equation}
where $\nabla_i\lambda^a, \nabla^*_i\lambda_a$ denote the canonical flat connections on $\underline\R^r$ or $(\underline\R^r)^*$; recall that we are identifying these two bundles. The vector bundle $E$ possesses the fiber metric
\begin{equation}\label{metric-E}
G\Bigl(\begin{pmatrix}
X^i \\
\lambda^a
\end{pmatrix}, \begin{pmatrix}
Y^j \\
\mu^b
\end{pmatrix}\Bigr)=g_{ij}X^iY^j+\lambda^a\mu_a,
\end{equation}
and $\boldsymbol{\nabla}, \boldsymbol{\nabla}^*$ are dual to each other with respect to this metric. By this duality, the curvature tensors $\boldsymbol{R}_{ij}{}^K{}_L, \boldsymbol{R}^*_{ij}{}^K{}_L$ of $\boldsymbol{\nabla}, \boldsymbol{\nabla}^*$ are related by
\[
\boldsymbol{R}_{ijKL}=-\boldsymbol{R}^*_{ijLK},
\]
where the index is lowered by the metric $G_{AB}$. Thus, it suffices to consider only the equation $\boldsymbol{R}_{ij}{}^K{}_L=0$. Since $\nabla$ is torsion-free, this equation is equivalent to
\[
(\boldsymbol{\nabla}_i\boldsymbol{\nabla}_j-\boldsymbol{\nabla}_j\boldsymbol{\nabla}_i)
\begin{pmatrix}
X^k \\
\lambda^a
\end{pmatrix}=0,
\]
where the covariant differentiation on $TM\otimes E$ is with respect to the coupled connection $\nabla\otimes \boldsymbol{\nabla}$. By the computation using \eqref{conn-E}, we have
\[
(\boldsymbol{\nabla}_i\boldsymbol{\nabla}_j-\boldsymbol{\nabla}_j\boldsymbol{\nabla}_i)
\begin{pmatrix}
X^k \\
\lambda^a
\end{pmatrix}=
\begin{pmatrix}
{\rm (i)} \\
{\rm (ii)}
\end{pmatrix},
\]
where
\begin{align*}
{\rm (i)}&=\bigl(R_{ij}{}^k{}_l-(h^*_{ai}{}^kh^a_{jl}-h^*_{aj}{}^kh^a_{il})\bigr)X^l 
+(\nabla_ih^*_{aj}{}^k-\nabla_jh^*_{ai}{}^k+h^*_{bi}{}^k\tau^b_{aj}-h^*_{bj}{}^k\tau^b_{ai})\lambda^a, \\
{\rm (ii)}&=-(\nabla_i h^a_{jl}-\nabla_j h^a_{il}+\tau^a_{bi}h^b_{jl}-\tau^a_{bj}h^b_{il})X^l \\
&\quad +(\nabla_i\tau^a_{bj}-\nabla_j\tau^a_{bi}+\tau^a_{ci}\tau^c_{aj}-\tau^a_{cj}\tau^c_{ai}-h^a_{il}h^*_{bj}{}^l+h^a_{jl}h^*_{bi}{}^l)\lambda^b.
\end{align*}
Hence the flatness of $\boldsymbol{\nabla}$ (and $\boldsymbol{\nabla}^*$) is equivalent to the following Gauss--Codazzi--Ricci equations:

\begin{equation}\label{GCR}
\begin{aligned}
&{\rm Gauss:} & &R_{ijkl}=h^*_{aik}h^a_{jl}-h^*_{ajk}h^a_{il}, \\
&{\rm Codazzi:} & &\nabla_{[i} h^a_{j]l}-h^b_{k[i}\tau^a_{bj]}=0, \quad 
\nabla_{[i}h^*_{aj]}{}^k+h^*_{b[i}{}^k\tau^b_{aj]}=0, \\
&{\rm Ricci:} & &\nabla_{[i}\tau^a_{bj]}+\tau^a_{c[i}\tau^c_{bj]}-h^a_{l[i}h^*_{bj]}{}^l=0.
\end{aligned}
\end{equation}
Note that the second Codazzi equation can also be written as
\[
\nabla^*_{[i}h^*_{aj]k}+h^*_{bk[i}\tau^b_{aj]}=0.
\]
\subsection{The statistical Bonnet theorem: proof of Theorem \ref{Bonnet}}
We will prove the Bonnet theorem for statistical manifold (Theorem \ref{Bonnet}). Suppose that an $n$-dimensional statistical manifold $(M, g, \nabla)$ admits tensors 
$h^a_{ij}, \ h^*_{aij}, \ \tau^a_{bi}$ satisfying the equations \eqref{GCR}. By Theorem \ref{equivalence}, it suffices to construct a local statistical embedding $f\colon M \longrightarrow \R^{n+r}\times(\R^{n+r})^*$ in the sense of Lauritzen.

We consider the vector bundle $E:=TM\oplus\underline\R^r$ and define linear connections $\boldsymbol{\nabla}, \boldsymbol{\nabla}^*$ on $E$ by \eqref{conn-E}.
By a direct computation, we see that these are dual to each other with respect to the fiber metric $G$ defined by \eqref{metric-E}. Moreover, the Gausss--Codazzi--Ricci equations \eqref{GCR} imply that both connections are flat. Hence we can take local frames $(e_A), (e^*_A)$ of $E$ which are parallel with respect to these connections: $\boldsymbol{\nabla}e_A=\boldsymbol{\nabla}^*e^*_A=0$. Since $G(e_A, e^*_B)$ is constant, we can choose $(e_A), (e^*_A)$ so that 
\begin{equation}\label{G}
G(e_A, e^*_B)=\d_{AB}.
\end{equation}
Let $F\colon E\longrightarrow M\times\R^{n+r}$ be the local trivialization by the frame $(e_A)$ and define a local $\R^{n+r}$-valued 1-form $\theta$ on $M$
by
\[
\theta(X):=F\Bigl(\begin{pmatrix}
X \\
0
\end{pmatrix}
\Bigr), \quad X\in\Gamma(TM).
\]
We will show that $\theta$ is a closed form. For $X, Y\in\Gamma(TM)$, we have
\[
F\Bigl(\begin{pmatrix}
\nabla_X Y \\
-h(X, Y)
\end{pmatrix}\Bigr)=
F\Bigl(\boldsymbol{\nabla}_X \begin{pmatrix}
Y \\
0
\end{pmatrix}\Bigr)=X\cdot\theta(Y).
\]
Since $h$ is symmetric, we have
\[
F\Bigl(\begin{pmatrix}
\nabla_X Y-\nabla_Y X \\
0
\end{pmatrix}\Bigr)=X\cdot\theta(Y)-Y\cdot\theta(X).
\]
It then follows that 
\begin{align*}
d\theta(X, Y)&=X\cdot\theta(Y)-Y\cdot\theta(X)-\theta([X, Y]) \\
&=\theta(\nabla_X Y-\nabla_Y X-[X, Y]) \\
&=0
\end{align*}
as $\nabla$ is torsion-free. Thus, $\theta$ is closed and locally we can write as $\theta=df$ with a mapping $f\colon M\longrightarrow \R^{n+r}$. Similarly, by using the local trivialization by the frame $(e^*_A)$, we define $\theta^*=d\varphi,\ \varphi\colon M\longrightarrow (\R^{n+r})^*$; in this case we regard the fiber as the dual space of $\R^{n+r}$. Then, from \eqref{G} we have
\[
\langle f_*X, \varphi_*Y\rangle
=\langle \theta(X), \theta^*(Y)\rangle
=G\Bigl(\begin{pmatrix}
X \\
0
\end{pmatrix}, \begin{pmatrix}
Y \\
0
\end{pmatrix}
\Bigr)
=g(X, Y).
\]
for $X, Y\in\Gamma(TM)$. We also have
\begin{align*}
\langle X\cdot f_*Y, \varphi_* Z\rangle&=\langle X\cdot\theta(Y), \theta^*(Z)\rangle \\
&=G\Bigl(\boldsymbol{\nabla}_X\begin{pmatrix}
Y \\
0
\end{pmatrix}, \begin{pmatrix}
Z \\
0
\end{pmatrix}
\Bigr)
\\
&=G\Bigl(\begin{pmatrix}
\nabla_X Y \\
-h(X, Y)
\end{pmatrix}, \begin{pmatrix}
Z \\
0
\end{pmatrix}
\Bigr) \\
&=g(\nabla_X Y, Z)
\end{align*}
for $X, Y, Z\in\Gamma(TM)$. Thus, $(f, \varphi)\colon M\longrightarrow \R^{n+r}\times (\R^{n+r})^*$ is a local statistical embedding in the sense of Lauritzen, and we complete the proof of Theorem \ref{Bonnet}.

\end{document}